\def\volno{27(1)}
\def\volyear{2020}
\def\papno{P1.45}

\documentclass[12pt]{article}

\makeatletter

\usepackage[a4paper]{geometry}
\usepackage{amsthm,amssymb}

\renewcommand{\ge}{\geqslant}
\renewcommand{\le}{\leqslant}

{\clearpage\newpage\section*{Corrigendum -- submitted #1}}{}

\ifx\volno\undefined\def\volno{25}\fi
\ifx\volyear\undefined\def\volyear{2018}\fi
\ifx\papno\undefined\def\papno{P00}\fi

\newcommand{\Copyright}[1]{\def\copy@right{\copyright\thinspace\ignorespaces#1}}
\newcommand{\dateline}[3]{\def\the@dateline{Submitted: \ignorespaces#1;
    Accepted: \ignorespaces#2; Published: \ignorespaces#3\\[0.5ex]\copy@right}}
\newcommand{\MSC}[1]{\def\the@MSC{\ignorespaces#1}}

\let\old@maketitle\maketitle
\renewcommand{\maketitle}{\date{\small\the@dateline}\old@maketitle}

\renewenvironment{abstract}{%
  \small
      \begin{center}%
	{\bfseries \abstractname\vspace{-.5em}\vspace{\z@}}%
      \end{center}%
      \quotation}%
  {\par\smallskip\noindent
   \textbf{Mathematics Subject Classifications: }\the@MSC\endquotation} 

\renewcommand\title[1]{\gdef\@title{\reset@font\Large\bfseries #1}}
\renewcommand\section{\@startsection {section}{1}{\z@}%
                                   {-3.5ex \@plus -1ex \@minus -.2ex}%
                                   {2.3ex \@plus.2ex}%
                                   {\normalfont\large\bfseries}}
\renewcommand\subsection{\@startsection{subsection}{2}{\z@}%
                                     {-3ex\@plus -1ex \@minus -.2ex}%
                                     {1.5ex \@plus .2ex}%
                                     {\normalfont\normalsize\bfseries}}
\renewcommand\subsubsection{\@startsection{subsubsection}{3}{\z@}%
                                     {-2.5ex\@plus -1ex \@minus -.2ex}%
                                     {1.5ex \@plus .2ex}%
                                     {\normalfont\normalsize\bfseries}}

\renewcommand\paragraph{\@startsection{paragraph}{4}{\z@}%
                                     {2ex \@plus.5ex \@minus.2ex}%
                                     {-1em}%
                                     {\normalfont\normalsize\bfseries}}

\renewcommand\subparagraph{\@startsection{subparagraph}{5}{\parindent}%
                                     {2ex \@plus.5ex \@minus .2ex}%
                                     {-1em}%
                                     {\normalfont\normalsize\bfseries}}

\newfont{\footsc}{cmcsc10 at 8truept}
\newfont{\footbf}{cmbx10 at 8truept}
\newfont{\footrm}{cmr10 at 10truept}
\renewcommand{\ps@plain}{%
\renewcommand{\@oddfoot}{\footsc the electronic journal of combinatorics
  {\footbf\volno} (\volyear), \#\papno\hfil\footrm\thepage}}
\g@addto@macro\bfseries{\boldmath}

  \newlength{\BiblioSpacing}
  \renewenvironment{thebibliography}[1]{%
    \begin{oldthebibliography}{#1}%
      \setlength{\parskip}{\BiblioSpacing}
      \setlength{\itemsep}{\BiblioSpacing}
  }%
  {%
    \end{oldthebibliography}%
  }

\theoremstyle{plain}
\newtheorem{theorem}{Theorem}
\newtheorem{lemma}[theorem]{Lemma}
\newtheorem{corollary}[theorem]{Corollary}
\newtheorem{proposition}[theorem]{Proposition}

\theoremstyle{definition}

\theoremstyle{remark}
\newtheorem{remark}[theorem]{Remark}

\newcommand{\doi}[1]{\href{http://dx.doi.org/#1}{\texttt{doi:#1}}}

\makeatother

\usepackage{amsmath,amssymb}
\usepackage[colorlinks=true,citecolor=black,linkcolor=black,urlcolor=blue]{hyperref}
\dateline{Mar 6, 2019}{Feb 3, 2020}{Feb 21, 2020}
\MSC{06E30, 05B15, 05B30}
\Copyright{  The authors. Released under the CC BY license (International 4.0).}

\title{On unbalanced Boolean functions with\\ best correlation immunity%
\thanks{This work was funded by the Russian Science Foundation under grant 18-11-00136.}}

\author{
Denis S. Krotov \qquad  Konstantin V. Vorob'ev\\
\small Sobolev Institute of Mathematics\\[-0.8ex]
\small Novosibirsk 630090, Russia\\
\small\tt \{krotov,vorobev\}@math.nsc.ru}

\def\wt{\mathrm{wt}}
\def\VQ{{\scriptstyle\mathrm{V}}Q}
\def\VVQ{{\scriptscriptstyle\mathrm{V}}Q}
\def\EQ{{\scriptstyle\mathrm{E}}Q}

\begin{document}

\maketitle

\begin{abstract}
It is known that the order of correlation immunity of a nonconstant unbalanced Boolean function in $n$ variables cannot exceed $2n/3-1$; moreover, it is $2n/3-1$ if and only if the function corresponds to an equitable $2$-partition of the $n$-cube with an eigenvalue $-n/3$ of the quotient matrix. The known series of such functions have proportion $1:3$, $3:5$, or $7:9$ of the number of ones and zeros. We prove that if a nonconstant unbalanced Boolean function attains the correlation-immunity bound and has ratio $C:B$ of the number of ones and zeros, then $CB$ is divisible by $3$. In particular, this proves the nonexistence of equitable partitions for an infinite series of putative quotient matrices. 

We also establish that there are exactly $2$ equivalence classes of the equitable partitions of the $12$-cube with quotient matrix $[[3,9],[7,5]]$ and $16$ classes, with $[[0,12],[4,8]]$. These parameters correspond to the Boolean functions in $12$ variables with correlation immunity $7$ and proportion $7:9$ and $1:3$, respectively (the case $3:5$ remains unsolved). This also implies the characterization of the orthogonal arrays OA$(1024,12,2,7)$ and  OA$(512,11,2,6)$.
\end{abstract}

\section{Introduction}\label{s:intro}

We study unbalanced Boolean functions with the maximum possible order of correlation immunity. 
A function $f: \{0,1\}^n \to \{0,1\}$ is called \emph{unbalanced} if the number of its ones is different from $0$, $2^{n-1}$, and $2^n$.
It is called \emph{$t$-th order correlation immune} if the number of ones (equivalently, zeros)
$(x_1,\ldots,x_n): f (x_1,\ldots,x_n)=1 $ is statistically independent on the values of any $t$ arguments.
Fon-Der-Flaass \cite{FDF:CorrImmBound} proved that the correlation-immunity order of an unbalanced Boolean function in $n$ variables cannot exceed 
$2n/3 - 1$; moreover, any unbalanced Boolean function $f$ of correlation-immunity order $2n/3 - 1$ corresponds
to an equitable $2$-partition of the $n$-cube $Q_n$ with  quotient matrix $[[a,b],[c,d]]$, 
where $a+b=c+d=n$ and $a-c=-n/3$
(a formal definition can be found in Section~\ref{s:def}; 
here, it is essential that the number of ones of $f$ relates to the number of zeros as $c:b$).
Nowadays, there are three known families of quotient matrices corresponding to such functions: $[[0,3T],[T,2T]]$, $[[T,5T],[3T,3T]]$
(found in \cite{Tarannikov2000}), $[[3T,9T],[7T,5T]]$ (found in \cite{FDF:12cube.en}). 
For each of the matrices 
$[[0,3],[1,2]]$, $[[1,5],[3,3]]$, and $[[0,6],[2,4]]$, a function
is unique up to equivalence.
Kirienko \cite{kirienko2002} found that there are exactly two  inequivalent unbalanced Boolean functions in $9$ variables  attaining the bound on the order of correlation immunity (the corresponding quotient matrix is $[[0,9],[3,6]]$).   
Fon-Der-Flaass \cite{FDF:12cube.en} started the investigation of the equitable partitions of $Q_{12}$ attaining 
the correlation-immunity bound.
It was shown that equitable partitions with quotient matrix $[[1,11],[5,7]]$
do not exist, while  equitable partitions with quotient matrix $[[3,9],[7,5]]$ were built
(see the construction in Section~\ref{s:3975}).
These results were also important from the framework of the study of parameters
of equitable $2$-partitions of the $n$-cube: 
they closed the smallest open cases remaining after the general paper \cite{FDF:PerfCol}.
After that, all quotient matrices of equitable $2$-partitions of the $n$-cube were characterized
for any $n$ smaller than $24$.
For $n=24$, the remaining questionable matrices were $[[1,23],[9,15]]$,  $[[2,22],[10,14]]$,  $[[3,21],[11,13]]$,  $[[5,19],[13,11]]$,  $[[7,17],[15,9]]$, and it is notable that all these parameters
correspond to unbalanced Boolean functions with extreme order of correlation immunity, $15=2n/3-1$.    

In the present work, 
we prove a new property of the equitable partitions that meet the 
correlation-immunity bound with equality. 
In particular, our results imply the nonexistence of an equitable partition with 
quotient matrix $[[2,22],[10,14]]$ or $[[5,19],[13,11]]$, as well 
as any Boolean function with correlation immunity $2n/3-1$ and proportion between 
the number of ones and the number of zeros $5:11$, $13:19$, or any $C:B$ such that 
$CB$ is not divisible by $3$. 
Besides that, we provide a characterization of all inequivalent equitable 
partitions with the quotient matrices $[[3,9],[7,5]]$ and $[[0,12],[4,8]]$.

From the theoretical point of view, studying Boolean functions lying on 
the correlation-immunity bound with different proportions of the number of ones and zeros
is the most intriguing part of our research. 
On the other hand, 
the functions of correlation-immunity order $2n/3-1$ with $2^{n-2}$ ones are of special interest 
because of the following two connections,
and our classification related with the quotient matrix $[[0,12],[4,8]]$
makes a contribution to their study.

The first connection is with $t$-resilient functions. 
A function $f:\{0,1\}^n\to \{0,1\}^m$ is called \emph{$t$-resilient} 
if for every $\bar a$ from $\{0,1\}^m$ the function 
\begin{equation}\label{eq:fa}
 f_{\bar a}(\bar x) = \left\{\begin{array}{ll} 
1 & \mbox{ if } f(\bar x)= \bar a\\
0 & \mbox{ if } f(\bar x)\ne\bar a
\end{array}\right.
\end{equation}
is correlation immune of order $t$ with $2^{n-m}$ ones.
The resilient functions are important for applications in cryptography, see e.g. \cite{Carlet:vect}.
If $m=2$, then $t\le 2n/3 -1$ \cite{Friedman:92}. If $m=2$ and $t = 2n/3 -1$, 
then the functions $f_{\bar a}$ belong to the class of functions we study and correspond to the equitable
partitions of $Q_n$ with quotient matrix $[[0,3T],[T,2T]]$, \nolinebreak $T=n/3$.

The second connection is with orthogonal arrays. 
An \emph{orthogonal array} OA$(N,n,2,t)$ (we consider only the binary orthogonal arrays) is a multiset of $N$ vertices on the $n$-cube such that
the number of its elements $(x_1,\ldots,x_n)$ with prescribed values in any $t$ positions
does not depend on those values, see e.g. \cite{HSS:OA}.
(Often, the elements of an orthogonal array 
are considered as being arranged as the rows or the columns 
of an $N\times n$ or $n\times N$ array, 
respectively).
An orthogonal array is \emph{simple} if it is an ordinary set, without multiplicities more than one.
It is straightforward that the simple OA$(N,n,2,t)$ 
are in one-to-one correspondence 
with the Boolean functions 
$\{0,1\}^n\to\{0,1\}$
of correlation-immunity order $t$ with $N$ ones (actually, the set of ones of such function forms the corresponding OA$(N,n,2,t)$). A result  of Bierbrauer \cite[Theorem~1]{Bierbrauer:95} says for OA$(N,n,q,t)$ that
\begin{equation}\label{eq:bier}
 N\ge q^n\left(1-\frac{(q-1)n}{q(t+1)}\right); 
\end{equation}
moreover, for a non-simple array the inequality becomes strict, which is straightforward from the proof, see \cite[p.\,181, line 4]{Bierbrauer:95} (note that for the simple binary orthogonal arrays, the bound was proved earlier by Friedman \cite[Theorem~2.1]{Friedman:92}).
The arrays OA$(2^{n-2},n,2,2n/3-1)$ lie on this bound;
hence, they are simple and, 
as follows from the results of \cite{FDF:CorrImmBound},
correspond to the equitable partitions with quotient matrix $[[0,3T],[T,2T]]$, $T=n/3$.
In particular, the results of our classification imply that there are 
exactly $16$ inequivalent 
OA$(1024,12,2,7)$ and exactly $37$ inequivalent OA$(512,11,2,6)$
(such arrays are obtained from OA$(1024,12,2,7)$ by shortening). 
The classification of orthogonal arrays with given parameters is a problem 
that attracts attention of many researchers, see the recent works \cite{BMS:2017:few}, \cite{BulRy:2018}, and the bibliography there. 
We note that the preceding computational results were successful 
for smaller $t$ and $N$ than in our case.
Futher discussion on the quotient matrix $[[0,3T],[T,2T]]$, related
equitable partitions, orthogonal arrays, and Boolean functions
can be found in~\cite{Krotov:ISIT2019:Resilient}.

With the other parameters considered in our paper, the situation is different.
The Fon-Der-Flaass bound was generalized to the binary
orthogonal arrays by Khalyavin \cite{Khalyavin:2010.en}, 
who proved that
any OA$(N,n,2,t)$ with $N<2^{n-1}$ satisfies $t\le 2n/3-1$. 
However, this does not mean that any array lying on this bound is simple
(e.g., there is a non-simple OA$(24,6,2,3)$ \cite{Taran:private2018}). 
The classification of non-simple orthogonal arrays that meet the Fon-Der-Flaass--Khalyavin bound is a separate problem, which is not considered in the current research.

\medskip

The introductory part of the paper continues with definitions and basic facts (Section~\ref{s:def}) and Section~\ref{s:tools}, where we describe the computer tools used for the classification results.
The main theoretical results of the paper are proved in Section~\ref{s:3-divisible}. 
Theorem~\ref{th_uniform_coordinates} states 
(in an equivalent formulation) that
if the correlation-immunity order of an unbalanced Boolean function $f$ lies on the Fon-Der-Flaass bound, then the number of ones of the derivative
$f^{(i)}(\bar x)=f(\bar x)+f(\bar x+\bar e_i)$ of $f$ in any basic direction $\bar e_i$, $i=1,\ldots,n$, does not depend on the direction. As a consequence, we have a new necessary condition on the existence of such functions and corresponding equitable $2$-partitions (Corollary~\ref{le_uniform_coordinates}).
Section~\ref{s:3975} 
contains the characterization of 
inequivalent equitable $2$-partitions of the $12$-cube 
with quotient matrix $[[3,9],[7,5]]$, 
based on the combination of theoretical and computational results, and a description of the original Fon-Der-Flaass construction \cite{FDF:12cube.en} of such partitions, including the representation via the Fourier transform.  
In Section~\ref{s:01248}, 
we describe the computational classification
of the equitable partitions of the $12$-cube
with quotient matrix $[[0,12],[4,8]]$.
The list of the all $16$ inequivalent partitions is given in the appendix. As we mentioned above, the last partitions correspond to the order-$7$ correlation immune   Boolean functions in $12$ variables with $2^{10}$ ones, and to the orthogonal arrays OA$(2^{10},12,2,7)$. In Section~\ref{s:21066}, we briefly discuss equitable partitions with quotient matrix $[[2,10],[6,6]]$ and the connection of such partitions with orthogonal
arrays OA$(1536,13,2,7)$.

\section{Definitions and basic facts}\label{s:def}
Let $G=(V,E)$ be an undirected graph. A partition $(C_0,\ldots, C_{\tau-1})$ 
of the vertex set $V$ 
into $\tau$ \emph{cells}
is called an \emph{equitable partition} 
(\emph{equitable $\tau$-partition})
with {\it quotient matrix} $M=(m_{ij})$ 
if for all $i,j \in \{0, \ldots, \tau-1\}$ any vertex of $C_i$ 
has exactly $m_{ij}$ neighbors in $C_j$.
Two partitions $(C_0,\ldots, C_{\tau-1})$ and $(C'_0,\ldots, C'_{\tau-1})$ of  $V$ 
are \emph{equivalent}
if there is a graph automorphism $\pi$ such that 
$\pi(C_i) \in \{C'_0,\ldots, C'_{\tau-1}\}$ for every 
$i$ from $\{0,\ldots,{\tau-1}\}$.
 
{The \emph{$n$-cube} $Q_n=(\VQ_n,\EQ_n)$} (also known as the Hamming graph $H(n,2)$)
is a graph whose vertices are the words of length $n$ over the alphabet $\{0, 1\}$,
also treated as vectors over the binary field GF($2$). 
Two vertices are adjacent if and only if they differ in exactly one coordinate position,
which is referred to as the \emph{direction} of the corresponding edge.
The \emph{Hamming distance} $d(\bar x,\bar y)$ between vertices $\bar x$ and $\bar y$ 
is the number of coordinates in which they differ.
The \emph{weight} $\wt(\bar x)$ of a word $\bar x$ is the number of ones in it.
By $(\bar x,\bar y)$, we denote the ordinary inner product of vectors: $(\bar x,\bar y)=x_1y_1+x_2y_2+\dots +x_ny_n$. 
For two vertices $\bar x=(x_1,\ldots,x_n)$, $\bar y=(y_1,\ldots,y_n)$, 
we will write $\bar x\preccurlyeq \bar y$ 
if $x_i \le y_i$ for all $i$ from $1$ to $n$.
We denote by $\bar e_i$ the word with all zeros except one $1$ in the $i$-th position; 
by $\overline 0$ and $\overline 1$, the all-zero and all-one words, respectively.  
For $\bar x, \bar y\in \VQ_n$, 
the set 
$\Gamma_{\bar x}^{\bar y}=\{\bar z+\bar y:\bar z \preccurlyeq \bar x\}$ 
is a {\it $k$-face} of $Q_n$,
where $k=\wt(\bar x)$
is the dimension of the face.

Any equitable $2$-partition of $Q_n$ satisfies the following necessary conditions
on the coefficients of its quotient matrix $[[a,b],[c,d]]$ \cite{FDF:PerfCol}, \cite{FDF:CorrImmBound}:
\begin{enumerate}
 \item[(a)]
 $a$, $b$, $c$, $d$ are nonnegative integers such that 
 $a+b=c+d=n$, $b>0$, $c>0$;
 \item[(b)] $\displaystyle\frac{b+c}{\gcd(b,c)}$ is a power of $2$;
 \item[(c)] if $b\ne c$, then $\displaystyle a-c\ge-\frac{n}{3}$.
\end{enumerate}
Condition (c) is a special case
of the bound $t\le 2n/3-1$
on the order $t$ of correlation immunity 
of an unbalanced Boolean function \cite{FDF:CorrImmBound}.
We say that an equitable $2$-partition with quotient matrix 
$[[a,b],[c,d]]$
\emph{attains the correlation-immunity bound}
if $b\ne c$ and $a-c = -{n}/{3}$ (equivalently, $b+c=4n/3$).
Besides (a)--(c), the only matrix for which the nonexistence
of the corresponding equitable $2$-partitions of $Q_n$ was established,
before the current research,
 was $[[1,11],[5,7]]$
 (usually, we also agree that $b\ge c$, because this can always be reached by choosing the order of cells), see \cite{FDF:12cube.en}.
 In Section~\ref{s:3-divisible}, we will prove a new necessary condition,
 which rejects the matrix $[[1,11],[5,7]]$ as well as an infinite number
 of other matrices satisfying (a)--(c).

Some of our results are formulated in terms of real-valued functions
defined on the vertices of the $n$-cube. 
Two such functions $f_1,f_2: \VQ_n\rightarrow \mathbb{R}$ are {\it equivalent} if there is a permutation $\pi$ of $n$ coordinate positions and a vector $\bar y$ such that $f_1(\bar y + \pi \bar x)=f_2(\bar x)$ for all $\bar x\in \VQ_n$. The \emph{norm} of a function $f$ is $\|f\|=(\sum_{\bar y\in \VVQ_n}{f(\bar y)^2})^\frac12$.

Given an eigenvalue $\lambda$ of the adjacency matrix of a graph $G=(V,E)$, a function $f:V \rightarrow \mathbb{R}$ is called an {\it eigenfunction} or a {\it $\lambda$-eigenfunction} of $G$ if it is not constantly zero and for every $x\in V$  $$\lambda\cdot f(x)=\sum_{y\in V:(x,y)\in E}f(y).$$ Note that the tuple of values of a $\lambda$-eigenfunction is essentially an eigenvector of the adjacency matrix of $G$ corresponding to the eigenvalue $\lambda$.

It is well known and easy to check that the eigenspectrum of $Q_n$ is $\{\lambda_i(n)=n-2i:\, i=0,1,\ldots,n\}$ and the set of functions $\{\chi_{\bar y}(\bar x)=(-1)^{(\bar x,\bar y)}:\, \wt(\bar y)=i\}$ is an orthogonal basis of the $\lambda_i(n)$-eigenspace of $Q_n$ for $i=0,1,\ldots, n$. 
Therefore, for a function $f$ defined on $\VQ_n$, the following identity holds $$f(\cdot)=\sum_{\bar y\in \VVQ_n}{\widehat{f}(\bar y)\chi_{\bar y}}(\cdot), $$
where $$\widehat{f}(\bar y)=\frac{1}{2^n}\sum_{\bar z\in \VVQ_n}{f(\bar z)(-1)^{(\bar z,\bar y)}}$$ 
is a {\it Fourier coefficient}, $\bar y\in \VQ_n$. 
By the weight of the coefficient $\widehat{f}(\bar y)$, we will understand the weight of $\bar y$.
The next properties of the basis $\{\chi_{\bar y}:\,\bar y\in \VQ_n\}$ follow instantly from its definition.
   \begin{proposition}\label{basis}
   For $\bar x,\bar y\in \VQ_n$ the following equalities hold:
   \begin{itemize}
    \item[\rm(i)] $\chi_{\overline 0}\equiv 1$,
         
        \item[\rm(ii)] $\chi_{\bar x}\chi_{\bar y}=\chi_{\bar x+\bar y}$.
   \end{itemize}
   \end{proposition}
   
We will need the following well-known properties of the basis functions.
\begin{proposition}
\begin{itemize}
    \item
[\rm(i)] For every $k\in\{0,\ldots,n-1\}$, every 
$\bar y\in \VQ_n$ of weight $n-k$, 
and every $(k+1)$-face $\Gamma$, it holds
$$\sum_{\bar x\in \Gamma}{\chi_{\bar y}(\bar x)}=0;$$
\item
[\rm(ii)] \textup{(see, e.g., \cite[Ch.\,5, Lemma~2]{MWS})} for every $\bar x$ and $\bar y$ from $\VQ_n$, it holds 
\begin{equation}\label{eq:Fur}
2^{n-\wt(\bar x)}\sum_{\bar z \preccurlyeq \bar x}{\widehat \chi_{\bar y}(\bar z)}
= 
\sum_{\bar z \preccurlyeq \bar x + \overline1}{\chi_{\bar y}(\bar z)}.
\end{equation}
   \end{itemize}
\label{p:sarkar}
      \end{proposition}
\begin{proof}

(i) Let $\Gamma=\Gamma_{\bar z}^{\bar z'}$ for some words $\bar z$ 
of weight $k+1$ and $\bar z'$. Since $\wt(\bar y)=n-k$, there is some coordinate position $j$ where $\bar z$ and $\bar y$ both have $1$. Thus, for every $\bar x\in \Gamma$, we have $\chi_{\bar y}(\bar x)+\chi_{\bar y}(\bar x+\bar e_j)=0$.

(ii) By the definition of a Fourier coefficient, $\widehat \chi_{\bar y}(\bar z)$ equals $1$ if $\bar z=\bar y$ and $0$ otherwise. Thus, the left side of (\ref{eq:Fur}) is equal to $2^{n-\wt(\bar x)}$ if $\bar y\preccurlyeq \bar x$ and zero otherwise. In the right side,
we also have 
$\sum_{\bar z \preccurlyeq \bar x + \overline1} 1 = 2^{n-\wt(\bar x)}$
if $\bar y\preccurlyeq \bar x$. 
In the remaining case 
$\bar y \succ \bar x$,
we have $0$ by arguments similar to (i).
\end{proof}

For a given set $V$ of $v$ elements, a {\it $(t,k,v)$-covering}, 
$t\le k\le v$, is a set $S$ of $k$-subsets of $V$ such that 
for every $t$-subset $T$ of $V$ there exists $K$ from $S$ such that $T\subseteq K$.
The following facts are trivial and well known, see e.g. \cite{Stinson:coverings}.

\begin{proposition}\label{coverings}
   If $S$ be a $(t,k,v)$-covering of a set $V$ of size $v$, then
\begin{itemize}
 \item 
   [\rm (i)] $|S|\ge   \frac{\binom{v}{t}}{\binom{k}{t}}$\textup;
    \item
   [\rm (ii)] for every $a\in V$, the set $S_{a}=\{K\backslash{\{a\}}:\, a\in K\in S\}$ is a $(t-1,k-1,v-1)$-covering of $V\setminus \{a\}$.
\end{itemize}
\end{proposition}

Given an equitable $2$-partition $(C_0,C_1)$ of $Q_n$ with quotient matrix $[[a,b],[c,d]]$,
by its \emph{associated function} we will understand the function $f:\VQ_n\rightarrow \mathbb{R}$ defined as follows: 
 
 $$f(\bar x)=\begin{cases}
 b, & \bar x\in C_0\\
 -c, & \bar x\in C_1.
 \end{cases}$$

\begin{lemma}[\cite{FDF:CorrImmBound,FDF:12cube.en}]\label{FourierSystem}
Let $(C_0,C_1)$ be an equitable $2$-partition of $Q_n$ with quotient matrix $[[a,b],[c,d]]$ and associated function $f:\VQ_n\rightarrow \mathbb{R}$.
 Then the following identities take place:
 \begin{eqnarray*}
   \widehat f(\bar x)&=&0 \qquad \mbox{for all $\bar x$ such that }\wt(\bar x)\neq \frac{b+c}{2}, \\
   (b-c)\widehat f(\bar x) &=& 
   \!\!\!\sum_{\bar y,\bar z:\,\bar y+\bar z=\bar x}\!\!\!
   \widehat f(\bar y)\widehat f(\bar z) \qquad \mbox{for all  } \bar x\neq \overline 0,\\
   bc&=& \sum_{\bar y}{\widehat f(\bar y)^2}.
  \end{eqnarray*}

\end{lemma}

\begin{proof}
	Counting  the values of $f$ over the neighbours of a given vertex, we find that $f$ is an $(n-b-c)$-eigenfunction of $Q_n$. Thus, all its nonzero Fourier coefficients have weight $\frac{b+c}{2}$.
	By the definition of the associated function, we know that $(f-b)(f+c)=0$.
	Therefore,  $$\left( \sum_{\bar y\in \VVQ_n}{\widehat{f}(\bar y)\chi_{\bar y}}-b\chi_{\overline 0}\right) \left( \sum_{\bar y\in \VVQ_n}{\widehat{f}(\bar y)\chi_{\bar y}}+c\chi_{\overline 0}\right) =0.$$ After removing parentheses and using Proposition \ref{basis}, we obtain the remaining equalities.
\end{proof} 

The \emph{kernel} of an equitable $2$-partition $C=(C_0,C_1)$ is the set 
$$\mathrm{ker}(C)=\{\bar y\in \VQ_n:  C_0=C_0+\bar y \}=\{\bar y\in \VQ_n: f(\bar x+\bar y)=f(\bar x) \mbox{ for all } \bar x \in \VQ_n \}$$ of all periods of the cells or, equivalently, of the associated function $f$.

\section{Computational tools}\label{s:tools}
\paragraph{Exact covering.} The approaches we apply for enumerating equitable partitions of $Q_{12}$ (the approaches are completely different for the quotient matrices $[[3,9],[7,5]]$ and $[[0,12],[4,8]]$) 
include solving instances of the exact covering problem.
In general, the exact covering problem can be formulated as follows.
Given elements $a_1$, \ldots, $a_k$,
natural numbers 
$\alpha_1$, \ldots, $\alpha_k$, 
and a collection 
$\mathcal A = \{A_1$, \ldots, $A_m\}$ 
of subsets of the set
$\{a_1$, \ldots, $a_k\}$,
find a subcollection $\mathcal A'$ 
of $\mathcal A$ such that each element
$a_i$ is contained in exactly $\alpha_i$
sets from $\mathcal A'$.
Most mathematical packages include methods 
for finding an exact cover in the case
$\alpha_1=\dots=\alpha_k=1$, 
which is solved much effectively in practice than the general problem.
However, our approaches need finding 
exact covers with different multiplicities.
We exploited the \texttt{libexact} package \cite{KasPot08}, which can be used in \texttt{c/c++} programs.

\paragraph{Isomorphism.}
To find the number of equivalence classes of $2$-partitions of the vertices of $Q_n$ from a considered class, or any intermediate objects, we use the standard technique described in  \cite[Sect.\,3.3]{KO:alg}. Namely, sets of vertices of $Q_n$ are represented by graphs in such a manner that two objects are equivalent if and only if the corresponding graphs are isomorphic. 
A famous package to work with the graph isomorphism 
is \texttt{nauty} \cite{nauty2014}.
The same approach allows to find the automorphism group of any object we deal with.

\paragraph{Double counting.}
The following nice approach, described in \cite[Sect.\,10.2]{KO:alg}, allows to partially validate the results of the exhaustive search.
Assume that we have finished the classification 
of some objects and have found 
a representative of every equivalence class. 
Knowing the order of the automorphism group 
of each representative, 
we can calculate the total number 
of different objects. 
If this number does not coincide 
with the number of objects 
found by the exhaustive search, 
then the search was erroneous.
This approach catches many kinds of 
systematic and random mistakes, 
but only works if the result 
of the search is not empty.
We checked the results of every step 
of our classification by this 
double-counting method.

\section{New necessary condition}\label{s:3-divisible}

In this section we provide a new necessary condition of the existence of equitable $2$-partitions of $Q_n$ attaining the bound \cite{FDF:CorrImmBound} on correlation immunity.
Given an equitable partition of a $n$-cube, we will say that an edge of the graph is \emph{composite} if it is incident to vertices from different cells of the partition.

\begin{theorem}\label{th_uniform_coordinates}
Let $(C_0,C_1)$ be an equitable partition of $Q_n$ with quotient matrix $[[a,b],[c,d]]$, $b\neq c$, attaining the correlation-immunity bound, i.e., $a-c=-\frac n3$. Let $f:\VQ_n\rightarrow \mathbb{R}$ be the function associated to this partition. The following statements are true:
\begin{itemize}
 \item[\rm (i)] the value $\displaystyle\sum_{\bar x:\,x_i=0} \widehat f(\bar x)^2$ does not depend on $i\in \{1,\ldots, n\}$\textup;
  \item[\rm (ii)] the number of composite edges of direction $i$
  does not depend on $i\in \{1,\ldots, n\}$.
\end{itemize}
\end{theorem}
\begin{proof}
(i) Since our partition attains the bound on correlation immunity, we have $a-c=-\frac{n}{3}$.
By Lemma \ref{FourierSystem}, we know that $\widehat f(\bar x)=0$ if $\wt(\bar x)\neq \frac{2n}{3}$,  and for every $\bar x \ne \overline 0$, the following equality holds:
$$ (b-c)\widehat f(\bar x) = \sum_{\bar y,\bar z:\,\bar y+\bar z=\bar x}\widehat f(\bar y)\widehat f(\bar z); 
\qquad \mbox{hence, }\
\widehat f(\bar x)^2 = \frac{1}{b-c}\sum_{\bar y,\bar z:\,\bar y+\bar z=\bar x}
\widehat f(\bar x)\widehat f(\bar y)\widehat f(\bar z).
$$
Take some $i\in \{1,\ldots, n\}$. 
Consider the square of the norm 
of the subfunction corresponding to $x_i=0$, where $\bar x=(x_1,\ldots,x_n)$.
Our goal is to show that this norm 
does not depend on the choice of $i$.
\begin{multline*}
\sum_{\bar x:\,x_i=0} \widehat f(\bar x)^2 =
\sum_{\bar x:\,x_i=0} \frac{1}{b-c} 
\sum_{\bar y,\bar z:\,\bar y+\bar z=\bar x}
\widehat f(\bar x)\widehat f(\bar y)\widehat f(\bar z) =
\frac{1}{b-c}
\sum_{\genfrac{}{}{0pt}{}{\bar x,\bar y,\bar z:}{\genfrac{}{}{0pt}{}{\bar x+\bar y+\bar z=\overline 0}{x_i=0}}}\!\!\!
\widehat f(\bar x)\widehat f(\bar y)\widehat f(\bar z) \\
=
\frac{1}{3(b-c)}
\Biggl(
\sum_{\genfrac{}{}{0pt}{}{\bar x,\bar y,\bar z:}{\genfrac{}{}{0pt}{}{\bar x+\bar y+\bar z=\overline 0}{x_i=0}}}\!\!\!
\widehat f(\bar x)\widehat f(\bar y)\widehat f(\bar z)
+
\!\!\!\sum_{\genfrac{}{}{0pt}{}{\bar x,\bar y,\bar z:}{\genfrac{}{}{0pt}{}{\bar x+\bar y+\bar z=\overline 0}{y_i=0}}}\!\!\!
\widehat f(\bar x)\widehat f(\bar y)\widehat f(\bar z)
+
\!\!\!\sum_{\genfrac{}{}{0pt}{}{\bar x,\bar y,\bar z:}{\genfrac{}{}{0pt}{}{\bar x+\bar y+\bar z=\overline 0}{z_i=0}}}\!\!\!
\widehat f(\bar x)\widehat f(\bar y)\widehat f(\bar z)
\Biggr).
\end{multline*}
We state that $$
\sum_{\genfrac{}{}{0pt}{}{\bar x,\bar y,\bar z:}{\genfrac{}{}{0pt}{}{\bar x+\bar y+\bar z=\overline 0}{x_i=0}}}
\!\!\!
\widehat f(\bar x)\widehat f(\bar y)\widehat f(\bar z)
+\!\!\!
\sum_{\genfrac{}{}{0pt}{}{\bar x,\bar y,\bar z:}{\genfrac{}{}{0pt}{}{\bar x+\bar y+\bar z=\overline 0}{y_i=0}}}
\!\!\!
\widehat f(\bar x)\widehat f(\bar y)\widehat f(\bar z)
+\!\!\!
\sum_{\genfrac{}{}{0pt}{}{\bar x,\bar y,\bar z:}{\genfrac{}{}{0pt}{}{\bar x+\bar y+\bar z=\overline 0}{z_i=0}}}
\!\!\!
\widehat f(\bar x)\widehat f(\bar y)\widehat f(\bar z)=\!\!\!
\sum_{\genfrac{}{}{0pt}{}{\bar x,\bar y,\bar z:}{\bar x+\bar y+\bar z=\overline 0}}
\!\!\!
\widehat f(\bar x)\widehat f(\bar y)\widehat f(\bar z).$$
Indeed, for every nonzero term $\widehat f(\bar x)\widehat f(\bar y)\widehat f(\bar z)$, each of the words $\bar x$, $\bar y$, $\bar z$ has exactly $\frac{n}{3}$ zeros and the positions of the zeros do not intersect for $\bar x$, $\bar y$, and $\bar z$ (which follows from $ \bar x+\bar y+\bar z=\overline 0$). Therefore, every nonzero summand $\widehat f(\bar x)\widehat f(\bar y)\widehat f(\bar z)$ in the right side occurs exactly in one sum in the left side of the equality. This observation proves the last equality and the first claim of the theorem.

(ii) Let us count the number of composite edges of an arbitrary direction $i \in \{1,\ldots, n\}$. Clearly, this value equals $$\frac{1}{2(b-c)^2}\sum_{\bar x \in \VVQ_n}{\big(f(\bar x + \bar e_i)-f(\bar x)\big)^2}.$$
Using the Fourier transform, we have that 
$$\sum_{\bar x \in \VVQ_n}
 {\big(f(\bar x + \bar e_i)-f(\bar x)\big)^2}
=\sum_{\bar x \in \VVQ_n}{\bigg( 2\!\!\!\sum_{\bar y \in \VVQ_n:\, y_i=1}\!\!\!\!\!\!
{\widehat f(\bar y)(-1)^{(\bar x,\bar y)}}\bigg)^2}.$$
After removing parentheses, dividing by $2^{n+2}$, and changing the order of summing, we have 
$$\frac{1}{2^n}
\sum_{\bar y \in \VVQ_n:y_i=1}
\sum_{\stackrel{\phantom{.}}{\bar y' \in \VVQ_n:y'_i=1}}
\!\!\!\!\!\!
{\widehat f(\bar y)\widehat f(\bar y')
\sum_{\bar x \in \VVQ_n}{(-1)^{(\bar x,\bar y+\bar y')}}}=
\!\!\!\!\!
\sum_{\bar y \in \VVQ_n:y_i=1} \!\!\!\!\!\! {\widehat f(\bar y)^2}=
\!\!\sum_{\bar y \in \VVQ_n}{\widehat f(\bar y)^2}-
\!\!\!\!\!\!
\sum_{\bar y \in \VVQ_n:y_i=0}\!\!\!\!\!\!{\widehat f(\bar y)^2}\!.$$
By claim (i), the proof is done. 
\end{proof}

\begin{corollary} \label{le_uniform_coordinates}
If there exists an equitable partition of $Q_n$ with quotient matrix  $[[a,b],[c,d]]$ 
attaining the correlation-immunity bound, $b \ne c$, then 
either $\frac{b}{\gcd(b,c)}$ or $\frac{c}{\gcd(b,c)}$ is divisible by $3$.
\end{corollary}
\begin{proof}
Let $(C_0,C_1)$ be an equitable $2$-partition of $Q_n$ with quotient matrix $[[a,b],[c,d]]$, $b\neq c$, attaining the correlation immunity bound. 
From the definition of an equitable partition, 
we see that there are $\frac{c}{b+c}2^n$ vertices in $C_0$. Consequently, there are exactly $\frac{bc}{b+c}2^{n-1}$ composite edges in the graph. By Theorem \ref{th_uniform_coordinates} we conclude that 
$$\frac{bc}{n(b+c)}2^{n-1}\in \mathbb{N}.$$
Since our partition attains the bound on correlation immunity, we have $a-c=-\frac{n}{3}$. The degree of the $n$-cube equals $n=a+b$; so, we have $n=\frac34(b+c)$. Substituting this expression to the number of edges, we prove the required statement.
\end{proof}

Corollary~\ref{le_uniform_coordinates} implies the nonexistence of an infinite sequence 
of putative parameters of equitable $2$-partition of $Q_n$ 
for which this question was open before. In particular, 
it gives an alternative proof of the nonexistence of equitable $2$-partitions of $Q_{12}$ 
with quotient matrix $[[1,11],[5,7]]$, 
which was shown in \cite{FDF:12cube.en}, 
and the nonexistence of $2$-partitions of $Q_{24}$ with quotient matrices 
 $[[2,22],[10,14]]$ and $[[5,19],[13,11]]$:
 
\begin{corollary}[example] \label{co_some_matrices}
There are no equitable $2$-partitions of $Q_n$ 
with quotient matrices \linebreak[4]
$[[T,11T],[5T,7T]]$ \textup($T=n/12$\textup) and $[[5T,19T],[13T,11T]]$ \textup($T=n/24$\textup). 
\end{corollary}

\section[Equitable partitions with quotient matrix {[[3,9],[7,5]]}]%
{The equitable partitions with quotient matrix $[[3,9],[7,5]]$}
\label{s:3975}
In this Section we characterize all inequivalent $2$-partitions of $Q_{12}$ with quotient matrix $[[3,9],[7,5]]$.
\subsection{General properties}
Let $(C_0,C_1)$ be an equitable $2$-partition with quotient matrix $[[3,9],[7,5]]$.
By direct counting, we have $|C_0|=7\cdot 256$ and $|C_1|=9\cdot 256$. Let $f$ be the associated function:
      $$f(\bar x)=\begin{cases}
      9, \,\,\, \bar x\in C_0\\
      -7, \,\,\,\bar x\in C_1.
      \end{cases}$$ 
By Lemma \ref{FourierSystem}, we know that $f$ is an eigenfunction corresponding to the eigenvalue $\lambda_8(12)=-4$ and all its nonzero Fourier coefficients have weight $8$. Therefore, Proposition~\ref{p:sarkar}(i) guarantees that the sum of values of $f$ over any $5$-face equals $0$. Consequently, any $5$-face contains exactly $18$ vertices from $C_1$ and $14$ vertices from $C_0$. Proposition~\ref{p:sarkar}(ii) gives us the identity
$$
16\cdot \widehat f(\bar x)
=
\sum_{\bar z  \preccurlyeq \bar x+\overline1 }{f(\bar z)} 
\qquad 
\mbox{ for all $\bar x$ such that $\wt(\bar x)=8$.}
$$

In the right side of the equality we have the sum of values of $f$ 
over some $4$-face of $Q_{12}$. This means that 
$\widehat f(\bar x)\in \{\frac{1}{16}(9m-7(16-m)):\,m=0,1,\ldots, 16\}=\{m-7:m=0,1,\ldots, 16\}$. 
In particular, $\widehat f(\bar x)$ is integer.

Let us take an arbitrary $\bar x$ of weight $9$ and use Proposition~\ref{p:sarkar}(ii) 
one more time:
$$ 
\sum_{\bar z \preccurlyeq \bar x}{\widehat f(\bar z)}
=
\frac18\sum_{\bar z \preccurlyeq \bar x + \overline1}{f(\bar z)}.
$$
Since the value from the right side of the equation belongs to $\{\frac18(9m-7(8-m)):m=0,1,\ldots, 8\}=\{2m-7:m=0,1,\ldots, 8\}$, the sum $\sum_{\bar z \preccurlyeq \bar x}{\widehat f(\bar z)}$ is odd.
For a given $\bar x\in \VQ_n$ of weight $9$,
there is at least one $\bar z \preccurlyeq \bar x$ 
of weight $8$ such that $\widehat f(\bar z)$ is odd.
In other words, the set of quadruples 
of zero coordinates 
of the weight-$8$ vertices $\bar z$
for which $\widehat f(\bar z)$ is odd forms a $(3,4,12)$-covering $T$. 
Our next goal is to describe the set of possible values $\widehat f$ can take.
     
Applying Lemma \ref{FourierSystem} to our function, 
we have 

\begin{eqnarray}
   \widehat f(\bar x)&=&0, \qquad\mbox{if }\wt(\bar x)\neq 8, \\ 
   2\widehat f(\bar x) &=& \sum_{\bar y,\bar z:\,\bar y+\bar z=\bar x}\widehat f(\bar y)\widehat f(\bar z),\qquad\mbox{if }\bar x\neq \overline 0,\label{eq:system_3975_2}\\ 
   \sum_{\bar x}{\widehat f(\bar x)^2}&=&63. \label{eq:system_3975_3}
\end{eqnarray}

Suppose there is $\bar y$ such that $|\widehat f(\bar y)|\ge 2$. Without loss of generality we take $$\bar y=(1,1,1,1,1,1,1,1,0,0,0,0).$$ 
By Theorem~\ref{th_uniform_coordinates} we have 
$$\sum_{\bar x:\,x_{12}=0} \widehat f(\bar x)^2=21.$$
By Proposition \ref{coverings}(ii),  
the elements of our covering $T$ containing the $12$-th coordinate position 
form a $(2,3,11)$-covering by odd values of $\widehat f$ of the set $\{1,2,\ldots, 11\}$. 
Since the sum of squares equals $21$ and $|\widehat f(\bar y)|\ge 2$, we conclude that the size of this covering is not bigger than $17$. By Proposition \ref{coverings}(i), it must be at least $19$, and we get a contradiction.

The  arguments above prove the following statement.
\begin{lemma}\label{+-1}
Let $f$ be the associated function of an equitable $2$-partition of $Q_{12}$ with quotient matrix $[[3,9],[7,5]]$. 
Then $\widehat f(\bar x)\in \{-1,0,1\}$
for every $\bar x\in \VQ_n$.
\end{lemma}



\subsection{Configurations of overcovered triples}
As follows from Lemma \ref{+-1} above,
for every triple $\{i,j,k\}$ of different coordinates the number of nonzeros
$\bar x=(x_1,...,x_{12})$ of $\widehat f$ such that $x_i=x_j=x_k=0$ is odd.
Since the nonzero $\bar x$ has exactly
$4$ zero coordinates (the quadruple of these coordinates will be referred to as \emph{block}), we have a covering of all triples by $63$ blocks, a $(3,4,12)$ covering. Consider the multiset
$A$ of all triples where the multiplicity of a triple is the number of blocks covering this triple (we know this number is odd). Reducing the multiplicities by $1$, we get a multiset $B$ with the coefficients equal to the number of ``overcovering'' of the corresponding triple. 
All these coefficients are
even, and hence we can divide them by $2$,
obtaining a multicet $C$.
The elements of $C$ will be called \emph{bitriples} 
(naturally, one bitriple in $C$ corresponds to two triples in $B$).
Taking into account the multiplicities,
we have exactly $16$ bitriples.
Indeed, $63$ blocks cover 
$4\cdot 63 = 252 = 220+2\cdot 16$ 
triples in total;
each of $12\cdot 11\cdot 10/3!=220$
$3$-subsets of the set of coordinates
is covered, plus each of $16$ bitriples
is covered twice.

\begin{lemma}
Every coordinate belongs to exactly
$4$ bitriples.
\end{lemma}

\begin{proof}
Every coordinate belongs to $21$ blocks, 
which cover $21\cdot3=63$ 
triples with given coordinate,
taking into account the multiplicities. 
The number of different such triples 
is $11\cdot 10/2 = 55$. 
So, we have $(63-55)/2$ bitriples
(recall that each of bitriples 
corresponds to two overcoverings,
by the definition).
\end{proof}

\begin{lemma}\label{l:2odd}
{\rm (i)} Every two different coordinates belong to an odd number of blocks,
{\rm (ii)} at least $5$.
\end{lemma} 

\begin{proof}
Assume that the 
$1$-st and $2$-nd coordinates 
meet in exactly $21-k$ blocks.
We know that the number of all blocks is $63$; exactly $63-42=21$ of them contain the first coordinate; exactly $21-k$ of them contain the first and the second coordinates. Hence, exactly $k$ blocks
contain the first coordinate 
and do not contain the second. 
Similarly, exactly $k$ blocks
contain the second coordinate and do not contain the first.

The sum of all $63$ values of
$f$ is $9$ or $-7$ (the value of $f$ in $\overline 0$). 

(a) The value of $f$ at $10\overline 0$ 
is also $9$ or $-7$;
therefore, among the $42$ nonzeros
with $1$ in the first coordinate,
either $a:=17$, or $a:=21$, or $a:=25$ values $-1$ and 
$25$, $21$, or $17$ values $+1$, respectively (for example, if $f(\overline 0)=-7$ and $f(10\overline 0)=9$, then among $42$ nonzeros 
with $1$ in the first coordinate,
$25$ should have the value $-1$
and $17$ the value $1$, 
for the sum change by $16$ 
during the sign inverse,
which corresponds to the translation of the partition by the vector $10\overline 0$).  

(b) The value of $f$ at $01\overline 0$ is  also $9$ or $-7$, 
therefore, among the $42$ nonzeros
with $1$ in the second coordinate,
either $b:=17$, or $b:=21$, or $b:=25$ values $-1$ and 
$25$, $21$, or $17$ values $+1$, respectively.

(c) The value of $f$ at $11\overline 0$ is  also $9$ or $-7$, 
therefore, among the $2k$ nonzeros
with different values in the first and 
the second coordinate,
either $c=k-4$, or $c=k$, or $c=k+4$ values $-1$ and 
$k+4$, $k$, or $k-4$ values $+1$, respectively.

In the arguments (a), (b), (c),
every nonzero occurs twice 
or does not occur at all 
(if it starts with $00$).
Indeed, if we denote by 
$\alpha_{i,j}$ the number of nonzeros
$\bar x=(x_1,...,x_{12})$ such that
$x_1=i$, $x_2=j$, and $f(\bar x) = -1$, then we get
$a=\alpha_{1,0}+\alpha_{1,1}$,
$b=\alpha_{0,1}+\alpha_{1,1}$,
$c=\alpha_{0,1}+\alpha_{1,0}$.
Hence, $a+b+c=2(\alpha_{1,0}+\alpha_{0,1}+\alpha_{1,1})$ is even. 
On the other hand, 
$a+b+c\in
\{k+30,k+34,k+38,k+42,k+46,k+50,k+54\}$.
It follows that $k$ is even and $21-k$ is odd.

(ii) follows from covering of all $10$ triples that include the given pair.
\end{proof}

\begin{corollary}\label{c:2even}\label{COR:BI2} 
Each two different coordinates belong to an even number of bitriples, $0$, $2$, or $4$.
\end{corollary}
\begin{proof}
Without loss of generality,
consider the first two coordinates.
every $4$-block containing them covers
exactly two triples they belong to.
So, the number of such $4$-blocks
is the half of the number of different
triples of form $\{1,2,i\}$, $i>2$,
plus the number, say $k$, of
bitriples of such form.
That is, $(12-2)/2+k$. By Lemma~\ref{l:2odd},
this number is odd. Hence, $k$ is even.
\end{proof}

Our next goal is to describe possible bitriple systems
up to equivalence.
We first assume that there is at least one 
bitriple of multiplicity $1$.

\begin{lemma}\label{l:8bi}
If there is a bitriple of multiplicity $1$, 
then it belongs to the collection of $8$ bitriples
$\{4\pm 3,5\pm 3,6\pm 3\}$,
up to a coordinate permutation.
\end{lemma} 
\begin{proof}
Without loss of generality assume that we have a bitriple
\underline{$\{1,2,3\}$} of multiplicity $1$. 
By Corollary~\ref{COR:BI2}, 
there is another bitriple with $1$ and $2$. 
Without loss of generality it is \underline{$\{1,2,9\}$}.
By Corollary~\ref{COR:BI2}, 
there is another bitriple with $1$ and $3$. 
It cannot be $\{1,9,3\}$, 
because in that case any choice of the forth bitriple 
with $1$ contradicts Corollary~\ref{COR:BI2}. 
So, it is \underline{$\{1,8,3\}$}, without loss of generality (we did not use $8$ before). By Corollary~\ref{COR:BI2}, the fourth element with $1$ is \underline{$\{1,8,9\}$}.

Again by Corollary~\ref{COR:BI2} 
and since the multiplicity of $\{1,2,3\}$ is $1$, 
there is another bitriple with $2$ and $3$.
If it is  $\{9,2,3\}$, then we have bitriples 
$\{1,2,3\}$, $\{1,2,9\}$, $\{9,2,3\}$ with $9$, 
and the fourth bitriple with $9$ 
contradicts Corollary~\ref{COR:BI2}.
A similar argument rejects $\{8,2,3\}$ 
(with respect to $3$).
So, without loss of generality, 
we have \underline{$\{7,2,3\}$}.

Now, 
the fourth bitriple with $3$ must be
\underline{$\{7,8,3\}$}; 
the fourth bitriple with $2$ must be
\underline{$\{7,2,9\}$}; 
the fourth bitriple with $5$ must be
\underline{$\{7,8,9\}$}.
\end{proof}

\begin{lemma}\label{l:16bi1}
If there is a bitriple of multiplicity $1$, 
then the multiset of bitriples is one of the following, up to a coordinate permutation:
\begin{eqnarray}\label{eq:16-1}
&&
\big\{\{4\pm 3, 5\pm 3, 6\pm 3 \},\ \{7\pm 3, 8\pm 3, 9\pm 3 \}\big\},
\\ \label{eq:16-2}
&&
\big\{\{4\pm 3, 5\pm 3, 6\pm 3 \},\ 
4{\cdot}\{4,5,6\},\ 4{\cdot}\{10,11,12\}\big\}, 
\\ \label{eq:16-3}
&&
\big\{\{4\pm 3, 5\pm 3, 6\pm 3 \},\ 
2{\cdot}\{4,5,9\pm 3\},\ 2{\cdot}\{10,11,9\pm 3\} \big\}, 
\\ \label{eq:16-4}
&&
\big\{\{4\pm 3, 5\pm 3, 6\pm 3 \},\ 
2{\cdot}\{4,5,6\},\ 2{\cdot}\{4,11,12\},\ 2{\cdot}\{10,5,12\},\ 2{\cdot}\{10,11,6\}\big\}. 
\end{eqnarray} 
\end{lemma}
\begin{proof}
By Lemma~\ref{l:8bi}, 
we have the first $8$ bitriples.
If there is another, $9$-th bitriple of multiplicity $1$, 
then by the same lemma we have \eqref{eq:16-1}.
If there is no $9$-th bitriple of multiplicity $1$,
then the remaining bitriples have multiplicity $2$ or $4$,
and a simple exhaust search results 
in \eqref{eq:16-2}--\eqref{eq:16-4}.
\end{proof}
If there is no bitriple of multiplicity $1$, 
then the multiplicities of bitriples are $2$ or $4$.
In this case, we can again divide them by $2$, 
which results in a multiset of $8$ triples, 
call them \emph{bibitriples}. 
Every coordinate is covered by exactly $2$ bibitriples.
So, if there is no bibitriples of multiplicity $2$, 
then the $8$ bibitriples form a $1$-$(12,3,2)$ design.
If there is exactly one bibitriple of multiplicity $2$,
the remaining $6$ form a $1$-$(9,3,2)$ design.
If there is exactly two bibitriples of multiplicity $2$,
the remaining $6$ form a $1$-$(6,3,2)$ design.
The remaining case is $4$ bibitriples of multiplicity $2$.
The number of $1$-$(v,2,2)$ designs is known for $v=12,9,6$, see \url{http://oeis.org/A110100}.
In particular, up to permutation of the coordinates,
we have $23$, $6$, and $2$ solutions, respectively.

Finally, we know that the multiset of bibitriples is one of $36 = 4+23+6+2+1$ equivalence classes.

\subsection[Coverings by 4-ples]{Coverings by $4$-ples}
For each multiset of bitriples,
we can find all possible systems of quadruples such that every triple is included $1+2m$ times, where $m$ is its multiplicity in the multiset of bitriples. To do this, we have to solve the corresponding instance of the exact cover problem. This can be done in seconds on a modern computer (we used the \texttt{libexact} \cite{KasPot08} package with \texttt{c++}). The result is as follows.

\begin{proposition}\label{p:180}
There are exactly $180$ equivalence classes of $(3,4,12)$
coverings such that the overcovered triples 
correspond to one of the $36$ equivalence classes 
of bitriples mentioned above. 
Only $5$ of $36$ equivalence classes of bitriples 
can be realized in this way; 
namely, \eqref{eq:16-1} 
\textup($112$ inequivalent coverings found\textup), \eqref{eq:16-2} 
\textup($1$ covering\textup), \eqref{eq:16-4} \textup($51$ coverings\textup),
$$4\times\{\{1,2,3\},\{4,5,6\},\{7,8,9\},\{10,11,12\}\}$$  
\textup($1$ covering\textup), and 
$$2\times\{
\{1,2,3\},\{1,5,6\},\{2,4,6\},\{3,4,5\},
\{7,8,9\},\{7,11,12\},\{8,10,12\},\{9,10,11\}\}
$$
\textup($15$ coverings\textup).
\end{proposition}

\subsection{Finding signs of the Fourier coefficients}
So, we have got $180$ candidates for the set of nonzeros. To find the Fourier coefficient in each nonzero, we will exploit equations \eqref{eq:system_3975_2}, \eqref{eq:system_3975_3}.
In particular, for $\bar x \ne \overline 0$, we have
\begin{eqnarray}
2\widehat f(\bar x)& =&  
\sum_{\bar y, \bar z: \, \bar y + \bar z=\bar x} \!\!\!
\widehat f(\bar y)\widehat f(\bar z),
\nonumber
\qquad
\mbox{or}\\
\quad
\widehat f(\bar x) &=& \!\!\!
\sum_{\bar y, \bar z:\, \bar y \prec \bar z,\,
\bar y + \bar z=\bar x} \!\!\!\!\!\!
\widehat f(\bar y)\widehat f(\bar z),\qquad
\mbox{in particular}\\ 
\quad
\widehat f(\bar x) &\equiv& \!\!\!
\sum_{\bar y, \bar z:\, \bar y \prec \bar z,\,
\bar y + \bar z=\bar x} \!\!\!\!\!\!
\widehat f(\bar y)\widehat f(\bar z)\bmod 2,
\label{eq:prec2}
\end{eqnarray}
where $\prec$ denotes lexicographic preceding.
The last equation 
immediately gives a necessary condition on the set of nonzeros (indeed, for every nonzero $\bar x$, we have $\widehat f(\bar x)\equiv 1 \bmod 2$; so, both parts of \eqref{eq:prec2} do not depend on the sign of $\widehat f$). This condition rejects $173$ of $180$ coverings, as shown in the following computational proposition.
\begin{proposition}[computational results]\label{p:7}
Among the $180$ coverings 
found in Proposition~\ref{p:180},
exactly $7$ coverings can correspond 
to the nonzeros of a $\{-1,0,1\}$-valued function $\widehat f$ satisfying  \eqref{eq:prec2}. All these $7$ coverings correspond to the set \eqref{eq:16-1} of bitriples.
\end{proposition}

Now assume that $F$ is the set of nonzeros of $\widehat f$ and that the function $\phi:F\to\{0,1\}$ defines the sign of $\widehat f$ in each nonzero:
\begin{equation}\label{eq:f-from-phi}
\widehat f(\bar x) = 
\left\{
\begin{array}{ll}
(-1)^{\phi(\bar x)} & \mbox{if }\bar x\in F, \\
0 & \mbox{if }\bar x \not\in F.
\end{array}
\right.
\end{equation}
We will show that the $63$ values of $\phi$ satisfy 
a system of $2^{12}-1$ linear equations over $GF(2)$, 
one equations for each $\bar x\ne \overline 0$.

Consider any zero $\bar x$ of $\widehat f$ 
different from $\overline 0$, i.e., $\widehat f(\bar x)=0$, $\bar x\ne \overline 0$.
By \eqref{eq:prec2},
the number of pairs
$\{\bar y, \bar z\}$ of elements from $F$
such that $\bar y + \bar z=\bar x$
is even (the pairs are unordered; 
so, we can always assume 
$\bar y\prec \bar z$).
Denote this number by $p(\bar x)$.
From \eqref{eq:prec2} we see that for 
$p(\bar x)/2$ pairs we have $\widehat f (\bar y)\widehat f (\bar z) = 1$
and for the rest $p(\bar x)/2$ pairs $\widehat f (\bar y)\widehat f (\bar z) = -1$.
It follows that the number of $-1$s among all such $x$s and $x$s 
has the same parity as $p(\bar x)/2$. Let us write this fact as an equation.
\begin{equation}\label{eq:in0}
\sum_{\bar y, \bar z\in F:\, \bar y \prec \bar z,\,\bar y + \bar z=\bar x}
(\phi(\bar y)+\phi(\bar z)) \equiv \frac{p(\bar x)}2 \bmod 2, \qquad \bar x\not\in F\cup\{\overline 0\}.
\end{equation}

Next, consider an arbitrary nonzero $\bar x \in F$. For simplicity assume that $\widehat f(\bar x)=1$. 
From \eqref{eq:prec2} we see that $p(\bar x)$ is odd, 
and we find from \eqref{eq:prec2} that the number of $-1$s 
among all considered $\widehat f(\bar y)$ and $\widehat f(\bar z)$
is $\frac{p(\bar x)-1}2$ if $\widehat f(\bar x)=1$ and $\frac{p(\bar x)+1}2$ if $\widehat f(\bar x)=-1$.
We derive the following identity.
\begin{equation}\label{eq:non0}\phi(\bar x)+
\sum_{\bar y, \bar z\in F:\, \bar y \prec \bar z,\,\bar y + \bar z=\bar x}
\left(\phi(\bar y)+\phi(\bar z)\right)  \equiv \frac{p(\bar x)-1}2 \bmod 2,\qquad \bar x\in F.
\end{equation}
We see that the $63$ values of $\phi$ satisfy the system of $4095$ equations \eqref{eq:in0}, \eqref{eq:non0} over the finite field
GF($2$) of order $2$ 
(some of the equations \eqref{eq:in0} are trivial, $0=0$; so, the actual system to solve has less than $800$ equations). 
This system can be solved for all of the $7$ remaining candidates for $F$.
\begin{proposition}[computational results]\label{p:3}
Among the $7$ sets considered in Proposition~\ref{p:7},
the system of equations \eqref{eq:in0}, \eqref{eq:non0} is consistent for exactly $2$ sets. In each of these $2$ cases, the rank of the system is $44$; so, the number of solutions is $2^{63-44}=2^{19}$.
\end{proposition}

It remains, 
among the $2^{19}$ solutions in each of $2$
cases, to choose the functions that correspond 
to the Fourier transform of 
$\{-7,9\}$-valued functions. 
It is doable in a reasonable time; 
however, the following observation 
reduces the number of calculations even more.

\begin{lemma}\label{l:transl}
For each $i$ from $1$ to $12$, 
define the coordinate function 
$\psi_i:F\to \{0,1\}$ 
by the identity $\psi_i(\bar v)=v_i$, 
where $\bar v=(v_1,\ldots,v_{12})$. 
\begin{itemize}
 \item [\rm (i)] If some $\phi: F\to \{0,1\}$ 
satisfies all equations 
\eqref{eq:in0}, \eqref{eq:non0}
then $\phi+\psi_i$ does.
\item [\rm (ii)]  Moreover, 
if $\widehat f$, see \eqref{eq:f-from-phi},
is the Fourier transform of 
a $\{-7,9\}$-valued function, 
then adding $\psi_i$ to $\phi$ does not change this property.
\end{itemize}
\end{lemma}
\begin{proof}
(i) It is easy to see that 
the number of $\bar y$s and $\bar z$s 
involved in \eqref{eq:in0}
that have $1$ in the $i$-th position
is even.
Indeed,
if $x_i=0$, then $y_i+z_i=0$ 
for every pair $(\bar y,\bar z)$ 
under the sum. 
If $x_i=1$, 
then $y_i+z_i=1$, 
but the number $p(\bar x)$
of the pairs involved in the sum is even.
Hence, adding $\psi_i$ does not change 
the parity of the left side of \eqref{eq:in0}.

The similar argument works for \eqref{eq:non0}
with the only difference that $p(\bar x)$ 
is even, which is compensated by involving 
$\bar x$ in the left side.

(ii) It is straightforward from the definition 
of the Fourier transform that the sum 
$\phi'=\phi+\psi_i$ corresponds to the 
translation $f'(\bar v)=f(\bar v+\bar e_i)$, 
where $\bar e_i$ has $1$ 
in the $i$-th position and $0$ in the others.
\end{proof}

So, the affine space of the all solutions $\phi$
can be partitioned into the cosets of the span 
$\langle \psi_1,\ldots,\psi_{12}\rangle$ 
(the span has dimension $11$ in one of the
remaining cases and dimension $10$ in the other),
and it is sufficient 
to test one representative
from every coset.
Finally we have found $6$ 
admissible representatives in one of the cases
and $12$ in the other. 
It occurs that in each of two cases, 
the equitable partitions found are equivalent.

\begin{theorem}\label{th:2}
There are exactly $2$ inequivalent 
equitable partitions of $Q_{12}$ 
with quotient matrix $[[3,9],[7,5]]$.
Each of them has the automorphism group 
of order $48$;
the sizes of orbits under the action 
of the automorphism group are 
$48^{30}$, $24^{14}$, $8^2$ 
for the smallest cell
\textup(in the notation 
$[\mbox{\rm orbit size}]^{[\text{\rm number of orbits}]}$\textup)
and $48^{40}$, $24^{16}$ for the largest cell.
One partition is coordinate transitive 
\textup(that is, the $12$ coordinates form 
one orbit under the action 
of the automorphism group\textup) 
and the size of its kernel is $2$.
The other partition 
has two coordinate orbits of size $6$ 
and the kernel of size $4$.
\end{theorem}

\subsection{Fon-Der-Flaass construction}
In this section, we define the equitable partitions constructed by Fon-Der-Flaass \cite{FDF:12cube.en} and describe the corresponding Fourier transforms.

First, we color the vertices of $Q_6$ into three colors as follows 
(symbol $*$ can be replaced by each of $0$ and $1$; 
so, a word like $0{**}100$ represent a set from $4$ vertices, 
which is referred to as a \emph{$2$-face}).

\begin{center}
\begin{tabular}{c|cccc}
 Black: 
& $000000$, & $111111$, & $000111$, & $111000$. \\ \hline
 White: 
& $100000$, & $011111$, & $000011$, & $111100$, \\
& $010000$, & $101111$, & $000101$, & $111010$, \\
& $001000$, & $110111$, & $000110$, & $111001$. \\ \hline
Gray: 
&$0{*}{*}100$,& $1{*}{*}011$, &$1001{*}{*}$, & $0110{*}{*}$, \\
$12$ 
&${*}0{*}010$,& ${*}1{*}101$, &$010{*}1{*}$, & $101{*}0{*}$, \\
$2$-faces
&${*}{*}0001$,& ${*}{*}1110$, &$001{*}{*}1$, & $110{*}{*}0$.\\
\end{tabular}
\end{center}

Next, color $Q_{12}$ as
$f_{12}(\bar u,\bar v):=f_6(\bar u+\bar v)$.

It remains to separate gray vertices into white and black. For every $2$-face $B$ from the twelve $2$-faces above, with $*$s in the $i$-th and $j$-th position, we color $(\bar u,\bar v)=(u_1,\ldots,v_6)$ such that $\bar u+\bar v\in B$ with respect to the parity $u_1+u_2+u_3+u_4+u_5+u_6+v_i+v_j$ (by black/white or white/black; so we have the choice for each $2$-face).
In such a way, 
we obtain $2^{12}$ different black/white colorings of the vertices of $Q_{12}$; the corresponding vertex partitions are equitable with quotient matrix $[[3,9],[7,3]]$ \cite{FDF:12cube.en}.
 
By Proposition \ref{p:sarkar}(ii) the Fourier coefficient at $\bar z$ (e.g., $\bar z=010011101111)$ is proportional (with $1/16$) to the sum of $f$ over the corresponding $4$-face (e.g., respectively, $*0{*}{*}000{*}0000$). 
So, the coefficients are straightforward to find. We omit technical details and describe the $2^{12}$ possibilities corresponding to the $2^{12}$ partitions constructed above. The nonzeros of one possible Fourier transform, with the corresponding signs, are the following:
\\
{\footnotesize
\newcommand\nnz[2]{\displaystyle \Big[\begin{array}{@{}c@{}}#1\\[-1mm]#2\end{array}\Big]}
\newcommand\nnzo[2]{\raisebox{2mm}{$\nnz{#1}{#2}$:}} 
\newcommand\nnza[2]{\raisebox{2mm}{$\nnz{#1}{#2}\!{+}$\!}}
\newcommand\nnzb[2]{\raisebox{2mm}{$\nnz{#1}{#2}\!{-}$\!}}
\newcommand\0{\boldsymbol{0}}\newcommand\1{\boldsymbol{1}}
\nnzo{\bar u}{\bar v}\quad
\nnzb{\0\01\,111}{\0\01\,111}\ 
\nnzb{\01\0\,111}{\01\0\,111}\ 
\nnzb{1\0\0\,111}{1\0\0\,111}\ 
\nnza{111\,\0\01}{111\,\0\01}\ 
\nnza{111\,\01\0}{111\,\01\0}\ 
\nnza{111\,1\0\0}{111\,1\0\0}\ 
\\\mbox{}\
\nnza{\011\,\011}{\011\,\011}\ %
\nnza{\011\,1\01}{\011\,1\01}\ %
\nnza{\011\,11\0}{\011\,11\0}\ %
\nnza{1\01\,\011}{1\01\,\011}\ %
\nnza{1\01\,1\01}{1\01\,1\01}\ %
\nnza{1\01\,11\0}{1\01\,11\0}\ %
\nnza{11\0\,\011}{11\0\,\011}\ %
\nnza{11\0\,1\01}{11\0\,1\01}\ %
\nnza{11\0\,11\0}{11\0\,11\0}\ 
\\\mbox{}\
\nnza{\0\0\0\,11\0}{\1\1\1\,11\1}\ 
\nnza{\0\0\1\,11\1}{\1\1\0\,11\0}\ 
\nnzb{\0\1\0\,11\1}{\1\0\1\,11\0}\ 
\nnzb{\0\1\1\,11\0}{\1\0\0\,11\1}\ 
\nnzb{\1\0\0\,11\1}{\0\1\1\,11\0}\ 
\nnzb{\1\0\1\,11\0}{\0\1\0\,11\1}\ 
\nnza{\1\1\0\,11\0}{\0\0\1\,11\1}\ 
\nnza{\1\1\1\,11\1}{\0\0\0\,11\0}\ 
\\\mbox{}\
\nnza{\0\0\0\,1\01}{\1\1\1\,1\11}\ 
\nnzb{\0\0\1\,1\11}{\1\1\0\,1\01}\ 
\nnza{\0\1\0\,1\11}{\1\0\1\,1\01}\ 
\nnzb{\0\1\1\,1\01}{\1\0\0\,1\11}\ 
\nnzb{\1\0\0\,1\11}{\0\1\1\,1\01}\ 
\nnza{\1\0\1\,1\01}{\0\1\0\,1\11}\ 
\nnzb{\1\1\0\,1\01}{\0\0\1\,1\11}\ 
\nnza{\1\1\1\,1\11}{\0\0\0\,1\01}\ 
\\\mbox{}\
\nnza{\0\0\0\,\011}{\1\1\1\,\111}\ 
\nnzb{\0\0\1\,\111}{\1\1\0\,\011}\ 
\nnzb{\0\1\0\,\111}{\1\0\1\,\011}\ 
\nnza{\0\1\1\,\011}{\1\0\0\,\111}\ 
\nnza{\1\0\0\,\111}{\0\1\1\,\011}\ 
\nnzb{\1\0\1\,\011}{\0\1\0\,\111}\ 
\nnzb{\1\1\0\,\011}{\0\0\1\,\111}\ 
\nnza{\1\1\1\,\111}{\0\0\0\,\011}\ 
\\\mbox{}\
\nnza{11\0\,\0\0\0}{11\1\,\1\1\1}\ 
\nnzb{11\1\,\0\0\1}{11\0\,\1\1\0}\ 
\nnza{11\1\,\0\1\0}{11\0\,\1\0\1}\ 
\nnzb{11\0\,\0\1\1}{11\1\,\1\0\0}\ 
\nnza{11\1\,\1\0\0}{11\0\,\0\1\1}\ 
\nnzb{11\0\,\1\0\1}{11\1\,\0\1\0}\ 
\nnza{11\0\,\1\1\0}{11\1\,\0\0\1}\ 
\nnzb{11\1\,\1\1\1}{11\0\,\0\0\0}\ 
\\\mbox{}\
\nnza{1\01\,\0\0\0}{1\11\,\1\1\1}\ 
\nnza{1\11\,\0\0\1}{1\01\,\1\1\0}\ 
\nnzb{1\11\,\0\1\0}{1\01\,\1\0\1}\ 
\nnzb{1\01\,\0\1\1}{1\11\,\1\0\0}\ 
\nnza{1\11\,\1\0\0}{1\01\,\0\1\1}\ 
\nnza{1\01\,\1\0\1}{1\11\,\0\1\0}\ 
\nnzb{1\01\,\1\1\0}{1\11\,\0\0\1}\ 
\nnzb{1\11\,\1\1\1}{1\01\,\0\0\0}\ 
\\\mbox{}\
\nnza{\011\,\0\0\0}{\111\,\1\1\1}\ 
\nnza{\111\,\0\0\1}{\011\,\1\1\0}\ 
\nnza{\111\,\0\1\0}{\011\,\1\0\1}\ 
\nnza{\011\,\0\1\1}{\111\,\1\0\0}\ 
\nnzb{\111\,\1\0\0}{\011\,\0\1\1}\ 
\nnzb{\011\,\1\0\1}{\111\,\0\1\0}\ 
\nnzb{\011\,\1\1\0}{\111\,\0\0\1}\ 
\nnzb{\111\,\1\1\1}{\011\,\0\0\0}.
}
\\
In each of the last six groups, all the signs can be inversed.
Additionally, in each of the last six groups, one can apply the coordinate permutation
$(4\ 10)(5\ 11)(6\ 12)$ to all $8$ nonzeros. The last transformation, applied to one group, switches between the two equivalence classes of the equitable partitions.






\section
[{[[0,12],[4,8]]} and related structures: classification]
{$[[0,12],[4,8]]$ and related structures: classification}
\label{s:01248}

The equitable partitions of the $12$-cube 
with quotient matrix $[[0,12],[4,8]]$ (or, equivalently, the orthogonal arrays OA$(1024,12,2,7)$, as was mentioned in the introduction) can be classified utilizing rather straightforward approach, a local exhaustive search, using the exact-covering software.
Let $S$ be the quotient matrix $[[0,12],[4,8]]$. We say that the pair of disjoint sets 
$P_0$, $P_1$ of vertices is an $r$-local (equitable) partition if $P_0 \cup P_1$ are the all words of weight at most $r$ and the neighborhood of every vertex of weight less than $r$ satisfy the local condition from the definition of the equitable partition.

So, there are exactly two $0$-local partitions, 
$(\{\overline 0\},\emptyset)$ and
$(\emptyset,\{\overline 0\})$.
For each of them, there is only one 
$1$-local partition, up to isomorphism. 

\begin{proposition}\label{p:2loc}
Up to isomorphism, there are exactly 
$94$ two-local partitions 
$(P_0,P_1)$ with $\overline 0\in P_0$,
and exactly $6$, with $\overline 0\in P_1$.
\end{proposition}
\begin{proof}
Let $\overline 0\in P_0$. In this case,
all weight-$1$ words are in $P_1$.
Consider the graph $\Gamma$ on the $12$ weight-$1$ words, where two vertices are adjacent
if in the $12$-cube they are adjacent 
to a common weight-$2$ word from $P_0$.
So, the weight-$2$ words from $P_0$ are
in one-to-one correspondence with 
the edges of $\Gamma$ (indeed, a weight-$2$ word has exactly $2$ weight-$1$ neighbors).
Next, we see that $\Gamma$ is a cubic graph
(indeed, every weight-$1$ word 
is in $P_1$ and hence has exactly $4$ neighbors 
from $P_0$; one of them is $\overline{0}$, 
the other $3$ correspond to edges of $\Gamma$).
The number of unlabelled connected cubic graphs
on $4$, $6$, $8$, and $12$ vertices
is $1$, $2$, $5$, $85$, respectively, see \url{http://oeis.org/A002851}.
So, the number of connected 
and disconnected cubic graphs on $12$ vertices is $85+5+3+1=94$. 

Let $\overline 0\in P_1$. Without loss of generality, all weight-$1$ words with $1$ in the first $8$ coordinates are assumed to be in $P_1$, the other $4$ in $P_0$. The last four words have no neighbors in $P_0$; so, any weight-$2$ word in $P_0$ has two weight-$1$ neighbors in $P_1$ and can be considered as an edge of some graph $\Gamma'$ on $8$ vertices (weight-$1$ words of $P_1$). From the quotient matrix, we see that  $\Gamma'$ is regular of degree $4$; so, its complement is cubic. There are $1$ disconnected and $5$ connected cubic graphs of order $8$.
\end{proof}

The search of the $3$-local partitions
was done by solving instances of the exact covering problem. We fix some $2$-local partition 
$(P_0,P_1)$ 
and consider the
weight-$2$ words in $P_1$ as the ``points''. 
To each ``point'' $\bar x$, 
we assign the multiplicity $\mu = 4-\lambda$, 
where $\lambda$ is the number of its weight-$1$ 
neighbors from $P_0$.
To each weight-$3$ word $\bar y$ that 
has no neighbors from $P_0$,
we assign a ``set'' $s(\bar y)$ of $3$ ``points'',
namely the $3$ weight-$2$ neighbors of $\bar y$.
With the chosen ``points'', their multiplicities,
and the ``sets'', we have an instance $\mathrm{Cov}(P_0,P_1)$ of the exact-covering problem.
Straightforwardly from  the definitions, we have the following one-to-one correspondence.

\begin{proposition}\label{p:loc-to-cov}
Given a $2$-local partition $(P_0,P_1)$,
the $3$-local partitions $(R_0,R_1)$ such that 
$P_0\subset R_0$ and $P_1\subset R_1$ 
are in one-to-one correspondence with the solutions 
$S$ of $\mathrm{Cov}(P_0,P_1)$. 
Namely, 
$S=\{s(\bar y)\mid \bar y \in R_0\backslash P_0\}.$
\end{proposition}

In such a way, for each of $94+6$ non-isomorphic 
$2$-local partitions, 
using \texttt{libexact}, we found all $3$-local continuations. After the isomorph rejection, we found all
non-isomorphic $3$-local partitions. 
The same approach allows to proceed the next step in finding the $4$-local partitions.
The results are checked using the double-counting approach (see Section~\ref{s:tools}). 

\begin{proposition}[computational results]\label{p:34loc}
The number of non-isomorphic $3$-local partitions
$(P_0,P_1)$ with $\overline 0\in P_0$ and 
$\overline 0\in P_1$ is $34$ and $222$, respectively.
For $4$-local partitions,
the number is $37$ and $81$, respectively.
\end{proposition}

The remaining part of the classification is based 
on the fact that the sum of the values of the 
$\{12,-4\}$-valued eigenfunction corresponding to 
a putative equitable partition 
(with considered parameters) over any $5$-face is zero.
Using this condition, one can uniquely reconstruct 
an eigenfunction by its values on the words of weight at most $4$ (actually, it is sufficient to know the values on the weight-$4$ words, see \cite[Theorem~3]{Vas2012en}).
It occurs that every $4$-local partition continues to a complete equitable partition (we have no theoretical proof of this fact). 

\begin{theorem}[computational results]\label{th:0-12-4-8}
There are exactly $16$ equivalence classes of equitable partitions $(P_0,P_1)$ of the $12$-cube with quotient matrix $[[0,12],[4,8]]$. In one of them, $P_0$ is a linear (or affine) subspace of the $12$-cube; two are ``full-rank'', i.e., the affine span of $P_0$ is the whole  $12$-cube; the other $13$ are ``semilinear'', that is, the affine span of $P_0$ consists of a half of the vertices of the $12$-cube. See the appendix for the list of representatives.
\end{theorem}

\begin{remark}
For the classification, 
it is sufficient to consider only the local partitions 
that meet $\overline 0\in P_0$, or only the local partitions that meet $\overline 0\in P_1$.
However, as the both ways were successful, 
we described in Propositions \ref{p:2loc} 
and \ref{p:34loc} the intermediate results 
for each of them. 
\end{remark}
\begin{remark}
The local search algorithm described in this section 
can be applied for finding equitable partitions with
different parameters (in different graphs).
However, we failed in the classification of the equitable
partition of the $12$-cube with quotient matrices 
$[[2,10],[6,6]]$ and $[[3,9],[7,5]]$ using the same approach. The corresponding instances of the exact-covering problem occur to be too large to solve with known tools. 
\end{remark}

The equitable partitions considered in the current section are related 
with several classes of combinatorial configurations.
The following lemma summarizes several known results
about such relations.
\begin{lemma} 
The objects from the following classes 
are in one-to-one correspondence:
\begin{itemize}
\item[\rm (I)]
the equitable partitions 
of the $n$-cube with quotient matrix 
$\left(
\begin{array}{cc}
 0 & n \\ c & n-c
\end{array}
\right),
$ $c<n$\textup;
\item[\rm (II)]
the orthogonal arrays OA$(N,n,2,t)$, 
where $t=\displaystyle\frac{n+c}{2}-1$ 
and $\displaystyle N=2^n\Big(1-\frac{n}{2(t+1)}\Big)$ 
\textup(so, the parameters attain the bound
\eqref{eq:bier}\textup{);}
\item[\rm (III)]
the orthogonal arrays OA$(N/2,n-1,2,t-1)$\textup;
\item[{\rm (IV)\makebox[0mm][l]{, (V)}}]
\mbox{}\phantom{, (V)}
the equitable partitions of the $(n-1)$-cube with quotient matrices \\
$
\left(
\begin{array}{ccc}
 0   & c-1 & n-c    \\  
 c-1 & 0   & n-c    \\ 
 c   & c   & n-2c-1 
\end{array}
\right)
$ and 
$
\left(
\begin{array}{ccc}
 c-1 & n-c    & 0   \\  
 c   & n-2c-1 & c   \\ 
 0   & n-c    & c-1  
\end{array}
\right),
$ respectively\textup;
\item[\rm (VI)]
the completely regular codes in $Q_{n-1}$ with the intersection array $(n-c,c; c,n-c)$.
\end{itemize}
\end{lemma} 

By the definition, a \emph{completely regular code} 
with the intersection array 
$(c_1,\ldots,c_r;$
$b_0,\ldots,b_{r-1})$ is a set of vertices
such that the distance partition with respect to it
is equitable with tridiagonal quotient matrix,
$(c_1,\ldots,c_r)$ and $(b_0,\ldots,b_{r-1})$ 
being the subdiagonal and the superdiagonal;
so, the correspondence between (V) and (VI) is straightforward. 
The connection between (I) and (II) is noted in \cite{Potapov:2010,Pot:2012:color}.
(III), (IV), and (VI) are related in \cite{Kro:OA13}.
It is known \cite[Proposition~2.3]{SeiZem:OA:66}
that for odd $t$, the arrays OA$(N/2,n-1,2,t-1)$ 
are in one-to-one correspondence with the self-complementary arrays
OA$(N,n,2,t)$ (a set $C$ of vertices of $Q_n$ is \emph{self-complementary}
if $C=C+\overline 1$);
on the other hand, 
the array of type~(II) 
must be self-complementary because
of the distance invariance 
of equitable partitions, 
see e.g.~\cite{Kro:struct}.


Once, for $n=12$ and $c=4$ 
we have representatives 
of all $16$ equivalence classes of partitions of type (I), 
it is rather straightforward
to find the number of equivalence classes 
of objects of types (II)--(VI). 
\begin{theorem}
 There are exactly $16$ inequivalent orthogonal arrays OA$(1024,12,2,7)$.
There are exactly $37$ inequivalent 
objects from each of the following 
families:
orthogonal arrays OA$(512,11,2,6)$;
completely regular codes in $Q_{11}$
with the intersection array
$(8,4;4,8)$;
equitable partitions of $Q_{11}$
with quotient matrices
$[[0,3,8],[3,0,8],[4,4,3]]$ and
$[[3,8,0],[4,3,4],[0,8,3]]$,
respectively.
\end{theorem}

\begin{remark}\label{rem:3to2}
Unifying the first two cells 
of a $3$-partition 
with quotient matrix
$[[0,3,8]$, $[3,0,8]$, $[4,4,3]]$,
we obtain an equitable 
$2$-partition with quotient matrix
$[[3,8],[8,3]]$. However, not all 
$2$-partitions with quotient matrix
$[[3,8],[8,3]]$ can be obtained in such a way.
\end{remark}
\section
[{[[2,10],[6,6]]}: discussion, connection with OA(1536,13,2,7)]
{$[[2,10],[6,6]]$: discussion, connection with OA$(1536,13,2,7)$}
\label{s:21066}

The remaining quotient matrix related 
with equitable partitions of $H(n,2)$
that attain the correlation-immunity bound 
and were not discussed in details above
is $[[2,10],[6,6]]$. The corresponding 
equitable partitions are of some special 
combinatorial interest 
because the first cell 
of such partition
induces a collection of disjoint cycles 
in the Hamming graph.
We failed to make the complete classification
of such partitions using
any approach described here directly.
However, the computational 
algorithm from Section~\ref{s:01248}
can be modified, dividing
the classification into more steps
depending on the value of some coordinate.
This way can be successful, 
but requires relatively 
large amount of computational
resources. 
Hopefully, 
the classification will be finished 
within several months,
and at this moment we can only 
announce that there are
more than $80$ equivalence classes 
of such partitions.
In this section, 
we briefly discuss the length of cycles 
induced by such a partition
and mention 
the relation with 
the quotient matrix $[[0,13],[3,10]]$, 
which corresponds to the orthogonal arrays
OA$(1536,13,2,7)$.

\subsection{Cycle lengths}\label{s:cycles}
As one can see from the first coefficient 
of the quotient matrix $[[2,10],[6,6]]$,
the first cell of a corresponding equitable
partition induces a regular subgraph 
of $Q_{12}$ of degree $2$, i.e.,
the union of disjoint cycles.
Any theoretical information 
on the structure of a partition 
could simplify
the classification; 
so, it is important to understand if the size 
of cycles is a constant or it can vary.
In the following proposition, 
we show that there
are partitions that induce 
both $4$-cycles 
and $8$-cycles. 
Another conclusion 
that can be made
from it is that 
the simple construction
\cite[Prop.~1(c)]{FDF:PerfCol}
that multiplies the quotient
matrix by an integer number 
can produce inequivalent
equitable partitions 
if one varies the addition,
treating the vertex set of the $n$-cube
as different $\mathbb Z_4$ modules.

\begin{proposition}\label{p:cycles}
For every $i$ from $\{0,1,2,3\}$,
 there is an equitable partition of $Q_{12}$ with quotient matrix $[[2,10],[6,6]]$
 such that the first cell induces $128\cdot (3-i)$ cycles of length $4$ and $64\cdot i$
  cycles of length $8$.
\end{proposition}
\begin{proof}
(i) We start with $i=0$ and construct a required partition $(D_0,D_1)$ using the doubling construction
\cite[Proposition~1(c)]{FDF:PerfCol}: if $(P_0,P_1)$ is an equitable partition 
 of $Q_{6}$ with  quotient matrix $[[1,5],[3,3]]$,
 then $(D_0,D_1)$ is defined by
 $$ D_i= \{(\bar x, \bar y)\mid \bar x+ \bar y \in P_i\}. $$
 Now consider an arbitrary vertex $(\bar x, \bar y)$ from $D_0$. 
 Let $\bar z =  \bar x+ \bar y\in P_0$, and let 
 $\bar z +\bar e_j$ be the only neighbor of $\bar z$ from $P_0$.
 Then $(\bar x, \bar y)$, 
 $(\bar x+\bar e_j, \bar y)$, 
 $(\bar x+\bar e_j, \bar y+\bar e_j)$, 
 $(\bar x, \bar y+\bar e_j)$ belong to $D_0$ and form a $4$-cycle. 
 So, every element of $D_0$ lies in a cycle of length $4$ with elements from $D_0$.
 
 (ii) Let $i=3$. Again, we  start with the partition $(P_0,P_1)$
 and use the same construction but with different addition:
 $$ D_i= \{(\bar x, \bar y)\mid \bar x\oplus \bar y \in P_i\}. $$
 Here, the sum
 of $(x_1,\ldots,x_6) \oplus(y_1,\ldots,y_6)$ is defined by pairs of coordinates:
 $(x_{2j-1},x_{2j})\oplus(y_{2j-1},y_{2j}) = \phi(\phi^{-1}(x_{2j-1},x_{2j})+\phi^{-1}(y_{2j-1},y_{2j}))$, 
 where $\phi$: $0 \to 00$, $1\to 01$, $2\to 11$, $3\to 10$
is the \emph{Gray map} from $\mathbb Z_4=\mathbb Z/4 \mathbb Z$ to GF$(2)^2 $.
It is straightforward that the construction doubled the parameters of the 
quotients matrices for this modification as well as in the case of usual addition.
Again, consider an arbitrary vertex $(\bar x, \bar y)$ such that 
 $\bar z =  \bar x+ \bar y$  and  
 $\bar z +\bar e_j$ are from from $P_0$.
 We see that $(\bar x, \bar y)$, 
 $(\bar x\oplus\bar e_j, \bar y)$, 
 $(\bar x\oplus\bar e_j, \bar y\ominus\bar e_j)$, 
 $(\bar x\oplus\bar e_j\oplus\bar e_j, \bar y\ominus\bar e_j)$, 
 $(\bar x\oplus\bar e_j\oplus\bar e_j, \bar y\ominus\bar e_j\ominus\bar e_j)$, 
 $(\bar x\ominus\bar e_j, \bar y\ominus\bar e_j\ominus\bar e_j)$, 
 $(\bar x\ominus\bar e_j, \bar y\oplus\bar e_j)$, 
 $(\bar x, \bar y\oplus\bar e_j)$ belong to $D_0$ and form a $8$-cycle. So, every element of $D_0$ lies in a cycle of length $8$ with elements from $D_0$.
 
 (iii) For arbitrary $i$, we use the same construction with the ``mixed'' addition, 
 the $\mathbb Z_4$ addition in the first $2i$ coordinates and the usual addition in the remaining ones.
The first cell $P_0$ of $(P_0,P_1)$ consists of $12$ edges, $2$ edges of each direction. 
The edges of the first $2i$ directions correspond to $8$-cycles in $D_0$, 
while the remaining edges
correspond to $4$-cycles. Counting the number of cycles is straightforward.
 \end{proof}


\subsection
[{[[0,13],[3,10]]} and OA(1536,13,2,7)]
{$[[0,13],[3,10]]$ and OA$(1536,13,2,7)$}
\label{s:013310}

The parameters of orthogonal arrays
OA$(1536,13,2,7)$ 
lie on the Bierbrauer--Friedman 
bound~\eqref{eq:bier}; 
hence, such arrays are simple.
Moreover, as was observed 
in \cite{Potapov:2010,Pot:2012:color},
any array on this bound corresponds 
to an equitable $2$-partition with 
the first coefficient 
of the quotient matrix being $0$. 
It is straightforward to see that 
the quotient matrix corresponding to 
OA$(1536,13,2,7)$ is
$[[0,b],[c,d]]=[[0,13],[3,10]]$ 
(indeed,
$0+b=c+d=13$ and $c:(b+c)=1536:2^{13}$);
equitable partitions with this quotient 
matrix are known to exist
\cite[Proposition~2]{FDF:PerfCol}
and moreover, the recent classification result
\cite{Kro:OA13} says that they are all equivalent.
By the argument similar to Remark~\ref{rem:3to2},
the ``projection'' of such a partition gives 
an equitable partition of $Q_{12}$ 
with quotient matrix $[[2,10],[6,6]]$. 
It occurs that only $3$ (of more than $80$)
inequivalent equitable partitions with quotient 
matrix $[[2,10],[6,6]]$ are related 
with OA$(1536,13,2,7)$ in such a way.
Further studying of the exceptional
properties of these three partitions
can potentially give a tip how to construct
other orthogonal arrays
attaining bound \eqref{eq:bier}
from equitable partitions.
For example, 
putative OA$(7\cdot 2^{20},25,2,15)$ 
are equivalent to putative 
equitable partitions with quotient 
matrix $[[0,25],[7,18]]$
and related 
to equitable partitions with quotient 
matrix $[[6,18],[14,10]]$,
which are known to exist
(this matrix is a multiple of $[[3,9],[7,5]]$, 
considered in the current paper);
so, one can try to construct 
OA$(7\cdot 2^{20},25,2,15)$
starting from $2\times [[3,9],[7,5]]$.

\appendix\section*{Appendix}\label{s:appendix}

\def\Aut{\mathrm{Aut}}
\def\Ker{\mathrm{Ker}}
\def\Rpr{\mathrm{Repr}}
\small
Below we list all $16$ inequivalent equitable partitions $(P_0,P_1)$ with quotient matrix $[[0,12],[4,8]]$. The parameters are listed in the following order: rank, 
i.e., the dimension of the affine span of $P_0$ ($10$, $11$, or $12$); the order of the automorphism group,
i.e., of the stabilizer of $P_0$ in $\Aut(Q_{12})$; 
the orbit sizes, for $P_0$, 
then for $P_1$; 
the subspace ${\Ker}$ (the ``kernel'', given by a basis) and a set ${\Rpr}$ (the set of representatives of cosets of the kernel) such that
$P_0=\{k+r|k\in{\Ker}, r \in{\Rpr}\}$ (the kernel ${\Ker}$ is the maximal subspace for which such decomposition is possible). The binary words of length $12$ are represented by hexadecimal numbers, e.g. 
$\mathrm{0a1}=0000\,1010\,0001$.
\begin{enumerate}
 \item Rank: 10, $|\Aut|=84934656$, orbits: $1024$; $3072$;\\
$\Ker$: $\langle$003, 005, 009, 030, 050, 090, 300, 500, 900, 111$\rangle$, \\
$\Rpr$:  \{000\}.

\item Rank: 11, $|\Aut|=1179648$, orbits: $1024$; $1024,2048$;\\
$\Ker$: $\langle$300, 500, 900, 111, 222, 444, 888, 00f$\rangle$, \\
$\Rpr$:  \{000, 003, 005, 081\}.

\item  Rank: 11, $|\Aut|=393216$, orbits: $1024$; $1024,2048$;\\
$\Ker$: $\langle$300, 500, 900, 111, 222, 444, 888, 003$\rangle$,\\
$\Rpr$:  \{000, 005, 009, 048\}.

\item  Rank: 11, $|\Aut|=147456$, orbits: $2 {\times} 128,768$; $ 2 {\times} 128,768,2 {\times} 1024$;\\
$\Ker$: $\langle$300, 500, 900, 111, 222, 444, 888$\rangle$,\\
$\Rpr$:  \{000, 003, 005, 006, 00a, 00c, 00f, 018\}.

\item  Rank: 11, $|\Aut|=49152$, orbits: $2 {\times} 512$; $ 2 {\times} 6,3 {\times} 512,1024$;\\
$\Ker$: $\langle$900, c00, 300, 444, 222, 099$\rangle$,
\\$\Rpr$:  \{000, 003, 005, 006, 00a, 00c, 017, 018, 030, 03c, 04b, 050, 0a0, 0c0, 188, 809\}.

\item  Rank: 11, $|\Aut|=24576$, orbits: $2 {\times} 512$; $ 2 {\times} 512,2 {\times} 1024$;\\
$\Ker$: $\langle$900, 300, 500, 144, 111, 0aa$\rangle$, 
\\$\Rpr$:  \{000, 003, 005, 006, 00a, 00c, 018, 01e, 027, 030, 060, 081, 096, 0c0, 488, 828\}.

\item  Rank: 11, $|\Aut|=196608$, orbits: $2 {\times}512$; $ 2 {\times}512,2048 $;\\
$\Ker$: $\langle$900, c00, 300, 033, 066, 0cc$\rangle$, 
\\$\Rpr$:  \{000, 003, 006, 00c, 012, 018, 048, 069, 224, 428, 4e1, 805, 809, 80a, 811, 814\}.

\item  Rank: 11, $|\Aut|=9216$,  orbits: $64, 3 {\times} 192, 384$; $ 64, 128, 3 {\times} 192, 6 {\times} 384$;\\
$\Ker$: $\langle$900, 300, 500, 144, 4bb$\rangle$,
\\$\Rpr$:  \{000, 003, 005, 006, 00a, 00c, 017, 018, 02e, 030, 035, 03c, 04b, 050, 059, 05a, 
\\\phantom{$\Rpr$:  \{}060, 069, 072, 081, 09c, 0a0, 0c0, 809, 811, 812, 821, 822, 828, 882, 888, 890\}.

\item  Rank: 11, $|\Aut|=24576$, orbits: $2 {\times} 256,512$; $ 2 {\times} 256,3 {\times} 1024$;\\
$\Ker$: $\langle$900, 300, 500, 0aa, 055$\rangle$,
\\$\Rpr$:  \{000, 003, 005, 006, 00a, 00c, 018, 027, 030, 036, 03c, 060, 06c, 081, 0b1, 0c0, 
\\\phantom{$\Rpr$:  \{}166, 2b4, 40f, 809, 811, 812, 814, 821, 822, 824, 828, 82d, 842, 848, 884, 890\}.

\item  Rank: 11, $|\Aut|=147456$, orbits: $1024$; $256,768,2 {\times} 1024$;\\
$\Ker$: $\langle$900, 300, 500, 847, 1b8$\rangle$, 
\\$\Rpr$:  \{000, 003, 005, 006, 00c, 018, 01b, 022, 02b, 02d, 030, 035, 048, 059, 05a, 060, 
\\\phantom{$\Rpr$:  \{}069, 071, 081, 08b, 090, 0c0, 809, 80a, 811, 812, 814, 821, 824, 850, 882, 884\}.

\item  Rank: 11, $|\Aut|=147456$, orbits: $256,768$; $ 256,768,2048$;\\
$\Ker$: $\langle$900, 300, 500, 0aa, 055$\rangle$,
\\$\Rpr$:  \{000, 003, 006, 00c, 00f, 012, 018, 021, 030, 036, 039, 048, 060, 081, 084, 0c0, 
\\\phantom{$\Rpr$:  \{}1b1, 21e, 2cc, 472, 496, 805, 809, 80a, 811, 814, 822, 824, 828, 82d, 842, 890\}.

\item  Rank: 11, $|\Aut|=18432$, orbits: $256, 2 {\times} 384$; $ 2 {\times} 256, 2 {\times} 384, 2 {\times} 768 $;\\
$\Ker$: $\langle$900, c00, 300, fff$\rangle$,
\\$\Rpr$:  \{000, 003, 006, 00c, 00f, 011, 017, 018, 028, 02b, 02d, 030, 035, 036, 03a, 044, 
\\\phantom{$\Rpr$:  \{}04b, 04d, 053, 056, 059, 05a, 05c, 060, 063, 066, 06a, 071, 081, 082, 09a, 0c0, 
\\\phantom{$\Rpr$:  \{}155, 178, 1b1, 1c6, 247, 2cc, 41b, 46c, 472, 805, 809, 80a, 812, 814, 81d, 81e, 
\\\phantom{$\Rpr$:  \{}821, 822, 824, 82e, 841, 842, 848, 850, 884, 888, 88b, 890, 896, 8a0, 8c3, 8d8\}.

\item  Rank: 11, $|\Aut|=6144$, orbits: $4 {\times} 128,2 {\times} 256$; $ 4 {\times} 64,6 {\times} 128,8 {\times} 256$;\\
$\Ker$: $\langle$900, c00, 300, fff$\rangle$,
\\$\Rpr$:  \{000, 003, 006, 00c, 00f, 011, 017, 018, 028, 02b, 02d, 030, 035, 036, 03a, 044, 
\\\phantom{$\Rpr$:  \{}04b, 04e, 053, 055, 059, 05a, 05c, 060, 063, 069, 072, 081, 082, 099, 09a, 0c0, 
\\\phantom{$\Rpr$:  \{}133, 1e4, 247, 256, 26a, 278, 427, 439, 46c, 4c3, 805, 809, 80a, 812, 814, 81d, 
\\\phantom{$\Rpr$:  \{}81e, 821, 822, 824, 82e, 841, 842, 848, 84d, 850, 871, 874, 884, 888, 890, 8a0\}.

\item  Rank: 11, $|\Aut|=18432$, orbits: $256,768$; $ 3 {\times} 256,3 {\times} 768$;\\
$\Ker$: $\langle$900, c00, 300, fff$\rangle$,
\\$\Rpr$:  \{000, 003, 006, 00c, 00f, 011, 018, 01d, 027, 028, 02d, 030, 033, 036, 03a, 044, 
\\\phantom{$\Rpr$:  \{}04b, 04e, 053, 055, 056, 059, 05a, 060, 063, 069, 06a, 081, 082, 08b, 09a, 0c0, 
\\\phantom{$\Rpr$:  \{}199, 21b, 21e, 22b, 235, 23c, 247, 278, 2a3, 46c, 472, 4b2, 805, 809, 80a, 812, 
\\\phantom{$\Rpr$:  \{}814, 817, 821, 822, 824, 82e, 841, 842, 848, 850, 871, 884, 888, 890, 8a0, 8c6\}.

\item  Rank: 12, $|\Aut|=32768$, orbits: $1024$; $1024,2048$;\\
$\Ker$: 
$\langle$111, 222, 444, 888, 003, 840$\rangle$,
\\$\Rpr$:  \{000, 04c, 009, 005, 054, 090, 030, 060, 051, 066, 06a, 01d, 03a, 036, 02c, 07c\}.

\item  Rank: 12, $|\Aut|=49152$, orbits: $1024$; $3072$;\\
$\Ker$: $\langle$00f, 0f0, f00, 333$\rangle$,\\
$\Rpr$: \{000, 
005, 050, 500,
550, 505, 055,
555,
021, 028, 041, 048,
210, 280, 410, 480,
\\\phantom{$\Rpr$:  \{}%
102, 802, 104, 804,
126, 146, 826, 846,
261, 461, 268, 468,
612, 614, 682, 684,
\\\phantom{$\Rpr$:  \{}%
016, 086, 206, 406,
160, 860, 062, 064,
601, 608, 620, 640,
111, 118, 181, 811,
\\\phantom{$\Rpr$:  \{}%
881, 818, 188, 888,
013, 083, 130, 830, 301, 308,
516, 586, 165, 865, 651, 658\}.
\end{enumerate}

\subsection*{Acknowledgements}
The authors thank Evgeny Bespalov, Vladimir Potapov, and Olli Pottonen for useful discussions and the referees for helpful comments.

 
\providecommand\href[2]{#2} \providecommand\url[1]{\href{#1}{#1}}

\end{document}